\documentclass[11pt]{amsart}

\pdfmapfile{}
\pdfmapfile{+cmextra.map}
\pdfmapfile{+symbols.map}
\pdfmapfile{+lm.map}

\usepackage[T1]{fontenc}
\usepackage{lmodern}
\usepackage[a4paper,margin=30mm]{geometry}
\usepackage{microtype}
\usepackage{mathtools}
\usepackage{amssymb}
\usepackage{enumitem}
\usepackage{xcolor}
\usepackage[colorlinks=true,linkcolor=blue!55!black,citecolor=blue!55!black,
  urlcolor=blue!65!black]{hyperref}
\usepackage[nameinlink,noabbrev]{cleveref}

\newtheorem{theorem}{Theorem}[section]
\newtheorem{lemma}[theorem]{Lemma}
\newtheorem{proposition}[theorem]{Proposition}
\newtheorem{problem}[theorem]{Problem}
\newtheorem{corollary}[theorem]{Corollary}
\newtheorem{fact}[theorem]{Fact}
\theoremstyle{definition}
\newtheorem{definition}[theorem]{Definition}
\theoremstyle{remark}
\newtheorem{remark}[theorem]{Remark}

\DeclareMathOperator{\Stab}{Stab}
\DeclareMathOperator{\acl}{acl}
\DeclareMathOperator{\tp}{tp}

\def\Ind#1#2{#1\setbox0=\hbox{$#1x$}\kern\wd0\hbox to 0pt{\hss$#1\mid$\hss}
\lower.9\ht0\hbox to 0pt{\hss$#1\smile$\hss}\kern\wd0}
\def\Notind#1#2{#1\setbox0=\hbox{$#1x$}\kern\wd0\hbox to 0pt{\mathchardef
\nn="3236\hss$#1\nn$\kern1.4\wd0\hss}\hbox to 0pt{\hss$#1\mid$\hss}\lower.9\ht0
\hbox to 0pt{\hss$#1\smile$\hss}\kern\wd0}
\def\ind{\mathop{\mathpalette\Ind{}}}

\newcommand{\U}{\mathbb U}
\newcommand{\Gm}{\Gamma}

\title[A counterexample to Petrykowski's conjecture]
  {A Counterexample to Petrykowski's Conjecture}
\date{27 August 2026}
\author{Artem Chernikov}

\begin{document}

\begin{abstract}
We construct a simple expansion of the theory of nonabelian free
groups in which the group admits a global type with bounded left-translation
orbit but is not definably amenable.  This extends the construction of Chernikov, Hrushovski, Kruckman, Krupi\'nski, Moconja, Pillay, and Ramsey, and gives a counterexample to  Petrykowski's conjecture.
\end{abstract}

\maketitle

\section{Introduction}

Petrykowski's conjecture states that if a definable group \(G\) admits a
global type \(p\in S_G(\U)\) whose orbit under left translation is bounded,
then \(G\) is definably amenable \cite[Conjecture~0.1]{Newelski}. Throughout, \(\U\) denotes a sufficiently saturated and strongly homogeneous
monster model of cardinality \(\bar\kappa\) with
\(2^\lambda<\bar\kappa\) for every \(\lambda<\bar\kappa\). Its
background lies in Newelski and Petrykowski's work on group coverings and
weakly generic types \cite{NPcover,NPweak}.  Petrykowski proposed that a
bounded orbit should suffice for definable amenability, and Newelski recorded
the conjecture and proved it under stronger hypotheses that
the space of global weakly generic types is absolutely bounded by $2^{\aleph_0}$ \cite{Newelski}. The case of fsg NIP groups provided an important precursor: invariant
measures there are obtained by combining a global generic type with Haar
measure on \(G/G^{00}\) \cite{hrushovski2008groups,conversano2012connected}.
 Conversano and Pillay proved the conjecture for groups definable in
o-minimal expansions of real closed fields
\cite[Corollary~4.12]{conversano2012connected}. It was proved for all NIP groups in \cite[Theorem~3.12]{CS}, where, under NIP, definable amenability is shown to be 
equivalent to the existence of a bounded orbit. It also holds (trivially) for groups definable in small theories, since every such
group is definably amenable \cite[Corollary~4.14]{CHK}.  The conjecture
nevertheless remained open in general outside NIP; see also
\cite{stonestrom2026f}.    In this paper we
show that it fails, even when the ambient theory is simple:
\begin{theorem}\label{thm:simple-main}
There  exist a complete simple first-order theory \(T\), a
 definable group \(G\) and a global type
\(p\in S_G(\U)\) such that $|G\cdot p| \leq 2^{\aleph_0}$, 
but \(G\) admits no left-invariant Keisler measure.  \end{theorem}

Our construction builds directly on the earlier construction of a non-definably amenable group in a simple theory \cite{CHK}. That construction
expands an algebraically closed field \(K\) by a generic good coloring of
\(G=\mathrm{SL}_2(K)\), encoded by four predicates
\(C_1,\ldots,C_4\).  It does not itself yield a bounded-orbit type; see
\Cref{rem: SL2 no bdd orbit}.  We retain the same good-coloring scheme and combinatorics (four colors,
with three prescribed left translates covering the group for each color), 
hence the same obstruction to a left-invariant Keisler measure, but generalize 
the corresponding model-companion and its simplicity theorems in \cite{CHK} to an arbitrary stable
nfcp group containing a non-commutative free subgroup (Theorem \ref{thm:generic-color}). For the counterexample, we apply this existence and simplicity of the model companion for naming a paradoxical decomposition to the theory of  a nonabelian free groups $F$ (which is stable and nfcp by \cite{SelaStability, SklinosNFCP}). Finite quotients of
a free group are then used to construct a global type $p$ stabilized by a subgroup of
bounded index as follows. Enumerate the quotient maps \(q_n:F\to R_n:=F/N_n\) associated with all
finite-index normal subgroups \(N_n\trianglelefteq F\), elementarily embed
the monster $\mathbb{U} \succ F$ of the group reduct into the ultrapower \(P^*=F^I/\mathcal D\), and put
\(Q^*=(F*\langle c\rangle)^I/\mathcal D\).  Writing
\(q_n^*:P^*\to R_n\) for the induced maps, let $
 H=\U\cap\bigcap_{n<\omega}\ker(q_n^*)$. 
Thus \(H\) consists of the monster elements invisible in every standard
finite quotient,  and we have an embedding of the quotient $\mathbb{U}/H$ 
into the countable product \(\prod_nR_n\), so \(H\) has bounded index.  For
\(h=[h_i]_{\mathcal D}\in H\), the coordinatewise  automorphism
defined by
\(\nu_{h_i}|_F=\mathrm{id}_F\) and \(\nu_{h_i}(c)=h_ic\)
fixes \(P^*\) pointwise, sends \([c]_{\mathcal D}\) to \(h[c]_{\mathcal D}\), and preserves every extended
finite-quotient map.  Residual finiteness makes the induced \(F\)-action on
the quotient by the group generated by these automorphisms free, so the good
coloring can be chosen invariant under them.  Relative homogeneity then gives
\([c]_{\mathcal D}\equiv_{\U}h[c]_{\mathcal D}\).  Thus \(H\leq\Stab(p)\) for
\(p=\tp([c]_{\mathcal D}/\U)\), and \(p\) has bounded orbit (see \Cref{sec:free-group} for the details).

The proof is organized as follows.  In \Cref{sec:good-colorings} we isolate
the finite extension test for good colorings of a free \(F_{12}\)-set from  \cite{CHK}.  In
\Cref{sec:stable-base} we axiomatize the model companion, prove an 
independent amalgamation lemma, and derive simplicity   
from it.  In \Cref{sec:free-group} we apply the result to the theory of
nonabelian free groups and use finite quotients to construct the bounded
orbit.

\subsection*{AI disclosure} 
ChatGPT 5.6 was used to establish 
 parts of the general model-companion construction and to find the bounded-orbit type 
construction.  The author subsequently
simplified and streamlined the arguments and presentation and takes full
responsibility for the contents.

\subsection*{Acknowledgements}
Chernikov was partially supported by the NSF
Research Grants DMS-2246598, DMS-2554164 and by the Alexander von Humboldt Foundation.

\section{Good colorings and a finite extension test}
\label{sec:good-colorings}
We recall some material from \cite[Section~3]{CHK}. 
Let \(\Gm=F_{12}\) be the free group of rank \(12\), with free generators $
  g_{ij}$ for $i\in[4],\ j\in[3]$, 
and let \(\Gm\) act freely on a set \(X\).  

\begin{definition}\cite[Definition 3.1]{CHK}.
	For \(A\subseteq X\), a map
\(c:A\to[4]\) is a \emph{good partial coloring} if, for every \(x\in A\) and
every \(i\in[4]\),
\[
  \bigl(g_{i1}x,g_{i2}x,g_{i3}x\in A\bigr)
  \quad\Longrightarrow\quad
  c(g_{ij}x)=i\text{ for some }j\in[3].                 \tag{2.1}
\]
\end{definition}

Fix \(n\geq1\), and put $
  \alpha(n)=2^{n+1}-1$. 
For an integer \(R\geq0\), let \(W_R\) be the reduced words of length at most \(R\) in the generators
and their inverses.  For an \(n\)-tuple \(\bar x=(x_1,\ldots,x_n)\), put
\[
  P_R(\bar x)=\{w x_r:w\in W_R,\ r\in[n]\}.
\]
Here \(1\in W_R\) denotes the empty word, and for \(A\subseteq X\) we write
\[
  B_R(A)=\{wa:w\in W_R,\ a\in A\}.
\]
Equivalently, \(B_R(A)\) is the radius-\(R\) ball about \(A\) in the graph
whose edges join \(x\) to \(g_{ij}^{\pm1}x\), and $
  P_R(\bar x)=B_R(\{x_1,\ldots,x_n\})$. 

\begin{fact}{\cite[Lemma 3.5]{CHK}}
\label{fact:finite-extension}
Suppose \(X_0\subseteq X\) has size \(m\geq1\), and let
\(c_0:X_0\to[4]\) be any map.  If there is a good partial coloring $
 c_B:B_{\alpha(m)}(X_0)\to [4]$ 
with \(c_B|_{X_0}=c_0\), then there is a total good coloring $
 c:X\to [4]$ 
with \(c|_{X_0}=c_0\).
\end{fact}

\begin{fact}[{\cite[Lemma~3.2]{CHK}}]
\label{fact:connected-extension}
Every good partial coloring of a nonempty connected subset of a free
\(\Gm\)-set extends to a total good coloring of the unique \(\Gm\)-orbit
containing it.
In particular, one may prescribe an arbitrary color at one chosen point of
an otherwise uncolored orbit.
\end{fact}

\begin{lemma}\label{lem:finite-test}
Let \(\bar x=(x_1,\ldots,x_n)\) be a tuple in a free \(\Gm\)-set \(X\),
and let \(\eta:[n]\to[4]\).  The prescription \(c(x_r)=\eta(r)\) extends to
a total good coloring of \(X\) if and only if $
  x_r=x_s \ \Rightarrow \ \eta(r)=\eta(s)$ 
and the resulting coloring of the distinct roots extends to a good partial
coloring of $
  P_{\alpha(n)}(\bar x)$. 
\end{lemma}

\begin{proof}
The forward implication follows by restricting a total good coloring to
\(P_{\alpha(n)}(\bar x)\): whenever all three relevant successors remain in
the restricted domain, the witness supplied by the total coloring remains
there as well.

Conversely, let \(X_0=\{x_1,\ldots,x_n\}\), and put \(m=|X_0|\).  The
compatibility condition makes the prescribed root coloring a well-defined
map \(c_0:X_0\to[4]\).  Since \(m\leq n\), one has
\(\alpha(m)\leq\alpha(n)\).  Restrict the assumed good partial coloring from
\(P_{\alpha(n)}(\bar x)=B_{\alpha(n)}(X_0)\) to
\(B_{\alpha(m)}(X_0)\); the restriction is still good for the same reason as
above.  The conclusion now follows from \Cref{fact:finite-extension}.
\end{proof}

\begin{lemma}\label{lem:gluing}
Let \((A_t)_{t\in I}\) be \(\Gm\)-invariant subsets of a free
\(\Gm\)-set \(X\).  Suppose  each \(A_t\) has a good coloring and that
these colorings agree on pairwise intersections.  Then they combine to a good
coloring of \(\bigcup_{t\in I}A_t\), and this coloring extends to a total good
coloring of \(X\).  Moreover, on each orbit disjoint from
\(\bigcup_{t\in I}A_t\), one may prescribe the color of one chosen point.
\end{lemma}

\begin{proof}
The colorings are compatible by assumption.  If a \(\Gm\)-orbit meets two
of the sets \(A_t\), invariance puts the whole orbit in their intersection.
It follows at once that the combined coloring is good.  On every remaining
orbit choose one point, with any prescribed color when one has been specified.
The one-point coloring is good because the action is free, and
\Cref{fact:connected-extension} extends it to a good coloring of the whole
orbit.  The union of the resulting orbitwise colorings is the required total
coloring.
\end{proof}

\section{Naming a generic paradoxical decomposition of a stable nfcp group}
\label{sec:stable-base}
%Recall that a formula \(\theta(x;y)\) has the \emph{finite cover property} if
%there are arbitrarily large finite inconsistent sets of instances of
%\(\theta\), every proper subset of which is consistent.  A theory has
%\emph{nfcp} if no formula
%has the finite cover property.  Equivalently, for every \(\theta(x;y)\) there
%is \(k<\omega\) such that an arbitrary set of instances of \(\theta\) is
%consistent whenever each of its subsets of size at most \(k\) is consistent.

Throughout this section, suppose \(T_0\) is a complete stable
nfcp \(L_0\)-theory with quantifier elimination,  \(G\) is an 
\(\emptyset\)-definable group, and the elements
\[
  a_{ij}\in G(\operatorname{dcl}_0(\emptyset))
  \qquad(i\in[4],\ j\in[3])
\]
freely generate the free group \(\Gamma\cong F_{12}\).  All word maps built from the
\(a_{ij}\) are parameter-free definable. Since \(T_0\) is complete,
every nontrivial reduced word in the \(a_{ij}\) is unequal to \(1\) in every
model of \(T_0\).  Consequently, left multiplication induces a free
\(\Gamma\)-action on \(G(M)\) for every \(M\models T_0\).

Fix a parameter-free formula
\(\delta_G(u)\) defining the domain of \(G\). For a
set \(A\) in a model, \(G(A)\) consists of the group tuples all of whose
coordinates lie in \(A\); \(\acl_0\) denotes algebraic closure, and \(A\ind^0_C B\) means that
\(\tp_{L_0}(A/BC)\) does not fork over \(C\) in $T_{0}$.  Every \(L_0\)-elementary map fixes the \(\emptyset\)-definable elements
\(a_{ij}\), and therefore commutes with every \(\Gamma\)-word map.  Consequently
\(G(A)\) is \(\Gamma\)-invariant whenever \(A\) is definably closed, and
transporting a good coloring of \(G(A)\) by an \(L_0\)-elementary map again
gives a good coloring.

\begin{definition}
	Let $
  L_{\mathrm{gc}}=L_0\cup\{C_1,C_2,C_3,C_4\}$, where \(C_i(u)\) are new relation symbols of the same sort as $G$. 
Let an $L_{\mathrm{gc}}$-theory 
\(T_0^{\mathrm{col}}\) say that its \(L_0\)-reduct is a model of \(T_0\),
that
\begin{gather*}
	\forall u\left(\delta_G(u)\longleftrightarrow
                 \bigvee_{i=1}^{4}C_i(u)\right),
 \qquad
 \forall u\bigwedge_{1\leq i<k\leq4}
                 \neg\bigl(C_i(u)\wedge C_k(u)\bigr), \textrm{ and}\\
                 \forall u\in G\ \bigvee_{j=1}^{3}C_i(a_{ij}u)
  \qquad(i\in[4]).                                           \tag{3.1}
\end{gather*}

\end{definition}
\noindent  Thus the restrictions of the \(C_i\) to \(G\) give a total good coloring for
the free left action of
\(\Gamma\) on $G$, in the sense of \Cref{sec:good-colorings}, with \(g_{ij}\) acting as left
multiplication by \(a_{ij}\).
The theory has models: for any \(M\models T_0\), apply
\Cref{lem:gluing} to the empty family of colored subsets of the free
\(\Gamma\)-set \(G(M)\).

\begin{lemma}\label{lem:measure-obstruction}
No model of \(T_0^{\mathrm{col}}\) admits a left-invariant Keisler measure on
the colored group \(G\).
\end{lemma}

\begin{proof}
For fixed \(i\), the three sets \(a_{ij}^{-1}C_i\), \(j\in[3]\), cover
\(G\).  Left invariance and finite subadditivity would therefore give $
  1\leq\sum_{j=1}^{3}\mu(a_{ij}^{-1}C_i)=3\mu(C_i)$, 
so \(\mu(C_i)\geq1/3\) for every \(i\in[4]\).  Since the four colors
partition \(G\), finite additivity would give $
  1=\sum_{i=1}^{4}\mu(C_i)\geq\frac43$, 
a contradiction.
\end{proof}

In the rest of the section we prove the following theorem (providing a partial justification for \cite[Remark 3.16]{CHK}):
\begin{theorem}
\label{thm:generic-color}
The theory
\(T_0^{\mathrm{col}}\) has a model companion 
\(T_0^{\mathrm{gc}}\).  For every completion of \(T_0^{\mathrm{gc}}\):
\begin{enumerate}[label=(\roman*)]
\item every color-preserving \(L_0\)-isomorphism between small
  \(\acl_0\)-closed sets is partial elementary;
\item algebraic closure is the same as algebraic closure in \(T_0\);
\item the theory is simple, and its nonforking relation coincides with non-forking in  $T_0$;
\item $G$ is not definably amenable (by Lemma \ref{lem:measure-obstruction}).
\end{enumerate}
\end{theorem}
\noindent (Here an \(L_0\)-map \(f:A\to B\) is \emph{color-preserving} if, for every
\(b\in G(A)\) and every \(i\in[4]\), $
 C_i(b) \Leftrightarrow C_i(f(b))$, 
where \(f\) is applied coordinatewise to the tuple representing \(b\).)
\subsection{The model companion}

Fix \(n\geq1\), \(\eta:[n]\to[4]\), and group variables
\(\bar z=(z_1,\ldots,z_n)\).  Put \(R=\alpha(n)\) and
\(I_R=W_R\times[n]\).  For \(p=(w,r)\in I_R\), let
\(\tau_p(\bar z)=wz_r\), where the abstract generators in \(w\) act by the
corresponding \(a_{ij}\).  For every map \(\lambda:I_R\to[4]\) satisfying
\(\lambda(1,r)=\eta(r)\), let
\[
 \operatorname{Coh}_\lambda(\bar z):=
 \bigwedge_{\substack{p,q\in I_R\\\lambda(p)\ne\lambda(q)}}
       \tau_p(\bar z)\ne\tau_q(\bar z),
\]
\[
\operatorname{Safe}_\lambda(\bar z) :=  \bigwedge_{\substack{p\in I_R\\k\in[4]}}
 \left[
   \left(
     \bigwedge_{j=1}^{3}\bigvee_{q\in I_R}
       \tau_q(\bar z)=a_{kj}\tau_p(\bar z)
   \right)
   \longrightarrow
   \left(
     \bigvee_{j=1}^{3}
     \bigvee_{\substack{q\in I_R\\\lambda(q)=k}}
       \tau_q(\bar z)=a_{kj}\tau_p(\bar z)
   \right)
 \right].
\]
Define
\[
 \operatorname{Good}_\eta(\bar z):=
 \bigvee_\lambda
 \bigl(\operatorname{Coh}_\lambda(\bar z)
       \wedge\operatorname{Safe}_\lambda(\bar z)\bigr),
 \qquad
 \operatorname{Good}_\emptyset:=\top.
\]
By assumption the word maps are $\emptyset$-definable, thus, by quantifier elimination,  \(\operatorname{Good}_\eta\) is a
quantifier-free \(L_0\)-formula.

\begin{lemma}[Formula for finite extendability]\label{lem:formula-test}
For \(M\models T_0\) and \(\bar z\in G(M)^n\), one has
\(M\models\operatorname{Good}_\eta(\bar z)\) if and only if the prescription
\(c(z_r)=\eta(r)\) extends to a total good coloring of the free
\(\Gm\)-set \(G(M)\).
\end{lemma}

\begin{proof}
The map $
 I_R \to P_R(\bar z)$ given by $(w,r)\mapsto wz_r$ 
is surjective.  The formula \(\operatorname{Coh}_\lambda\) says that
\(\lambda\) is constant on its fibers, so it induces a coloring of
\(P_R(\bar z)\) with the prescribed root colors.  The formula
\(\operatorname{Safe}_\lambda\) says that the induced coloring is a
good partial coloring.  The result is now \Cref{lem:finite-test}, applied to
the ambient free \(\Gm\)-set \(G(M)\).\end{proof}

\begin{lemma}
\label{lem:existential-normal-form}
Modulo \(T_0^{\mathrm{col}}\), every existential
\(L_{\mathrm{gc}}\)-formula \(\exists\bar x\,\theta(\bar x;\bar y)\), where $\theta(\bar x;\bar y)$ is quantifier-free, is equivalent to a finite disjunction of formulas $
  \exists\bar x\,\exists\bar z\,
  \varphi(\bar x,\bar z;\bar y)$, 
where 
\[
 \varphi(\bar x,\bar z;\bar y)
 =\psi_0(\bar x,\bar z;\bar y)
   \wedge\bigwedge_{r\in[n]}C_{\eta(r)}(z_r),               \tag{3.3}
\]
for some $0 \leq n <  \omega$, $\bar{z} = (z_1, \ldots,z_n)$ with $z_r$ of same sort as $G$,  $\eta: [n] \to [4]$ and \(\psi_0\) a quantifier-free \(L_0\)-formula implying 
\(\delta_G(z_r)\) for every \(r\in[n]\).  
\end{lemma}

\begin{proof}
Put \(\theta\) in disjunctive normal form.  In
\(T_0^{\mathrm{col}}\), every negative color literal can be replaced using $
 \neg C_i(t) \Leftrightarrow 
 \neg\delta_G(t)\ \vee\!\bigvee_{k\ne i}C_k(t)$. 
After distributing the resulting disjunctions, write $
 \theta(\bar x;\bar y) \Leftrightarrow
 \bigvee_{\ell<m}\theta_\ell(\bar x;\bar y)$, 
where each \(\theta_\ell\) is a conjunction of \(L_0\)-literals and positive
color literals.  Fix \(\ell<m\).  For every color literal
\(C_{\eta(r)}(t_r(\bar x,\bar y))\) occurring in \(\theta_\ell\), introduce a
new existentially quantified group variable \(z_r\), add
\(z_r=t_r(\bar x,\bar y)\) and \(\delta_G(z_r)\) to the \(L_0\)-part, and
replace the literal by \(C_{\eta(r)}(z_r)\).  The conjunction of the original
\(L_0\)-literals, these equalities, and the \(\delta_G\)-conditions is the
formula \(\psi_0\) in (3.3).  Thus
\(\exists\bar x\,\theta_\ell(\bar x;\bar y)\) is equivalent to
\(\exists\bar x\,\exists\bar z\,
  \varphi(\bar x,\bar z;\bar y)\).  Taking the disjunction over
\(\ell<m\) proves the lemma.
\end{proof}

Consequently, in axiomatizing existentially closed models it suffices to
consider formulas \(\varphi\) of the normalized form (3.3).  

\begin{lemma}\label{lem:definable-avoidance}
Let \(\bar z=(z_1,\ldots,z_n)\) be a tuple of group variables, 
\(S\subseteq[n]\) and \(\bar z_S=(z_{s} : s \in S)\).  For every \(L_0\)-formula
\(\beta(\bar x,\bar z;\bar y,\bar u)\), there is an
\(L_0\)-formula \(\operatorname{Av}_{\beta,S}(\bar y,\bar u)\) such that,
for any \(M\models T_0\) and \(\bar m,\bar c\in M\),
\(M\models\operatorname{Av}_{\beta,S}(\bar m,\bar c)\) if and only if some
elementary extension \(N\succcurlyeq_{L_0} M\) contains a realization of
\(\beta(\bar x,\bar z;\bar m,\bar c)\) with
\(z_s\notin G(M)\) for every \(s\in S\).  Equivalently, the partial type
\[
 \{\beta(\bar x,\bar z;\bar m,\bar c)\}
 \cup\{z_s\ne v:s\in S,\ v\in G(M)\}                     \tag{3.4}
\]
is consistent over the elementary diagram of \(M\).
\end{lemma}

\begin{proof}
If \(S=\emptyset\), take \(\operatorname{Av}_{\beta,S}\) to be
\(\exists\bar x\,\exists\bar z\,\beta\).  Otherwise let 
\[
 \chi(\bar x,\bar z;\bar y,\bar u,\bar v):=
 \beta(\bar x,\bar z;\bar y,\bar u)
 \wedge\bigwedge_{s\in S}z_s\ne v_s,
\]
where $\bar{v} = (v_s : s \in S)$ is a tuple of group variables. 
By nfcp of $T_0$, there is \(k\geq1\) such that any set of instances of \(\chi\) is
consistent if and only if each of its subfamilies of size at most \(k\) is
consistent.  Define
\[
 \operatorname{Av}_{\beta,S}(\bar y,\bar u):=
 \forall\bar v^1\in G^{|S|}\cdots\forall\bar v^k\in G^{|S|}\ 
 \exists\bar x\,\exists\bar z\,
 \left(
   \beta(\bar x,\bar z;\bar y,\bar u)
   \wedge
   \bigwedge_{\ell=1}^{k}\bigwedge_{s\in S}z_s\ne v_s^\ell
 \right).
\]
 As \(\bar v\) ranges over \(G(M)^{|S|}\), the
resulting family of instances of \(\chi\) is equivalent to (3.4): each
individual inequality in (3.4) occurs in such a block after arbitrary values
are chosen for the other coordinates.  And its finite
subfamilies are consistent over the elementary diagram of \(M\) exactly when
the corresponding existential formulas hold in \(M\). 
\end{proof}

Fix an \(L_{\mathrm{gc}}\)-formula \(\varphi(\bar x,\bar z;\bar y)\) of the normalized form
(3.3), with $\bar{z} = (z_1, \ldots, z_n)$.  For \(I\subseteq[n]\), all \(I\)-indexed tuples below, such as
\(\bar z_I=(z_i)_{i\in I}\), are written in increasing order of the
indices, and \(\eta_I:[|I|]\to[4]\) denotes the corresponding reindexing
of \(\eta|_I\).  For \(J\subseteq[n]\), let \(S=[n]\setminus J\), write
\(\bar u_J=(u_r)_{r\in J}\) using this convention, and set
\[
 \beta_J(\bar x,\bar z;\bar y,\bar u_J):=
 \psi_0(\bar x,\bar z;\bar y)
 \wedge\bigwedge_{r\in J}z_r=u_r
 \wedge\operatorname{Good}_{\eta_S}(\bar z_S).
\]
Define the \(L_{\mathrm{gc}}\)-formula 
\[
 \operatorname{Ext}_\varphi(\bar y):=
 \bigvee_{J\subseteq[n]}\exists\bar u_J
 \left[
   \bigwedge_{r\in J}
      C_{\eta(r)}(u_r)
   \wedge
   \operatorname{Av}_{\beta_J,S}(\bar y,\bar u_J)
 \right].                                                   \tag{3.5}
\]

\begin{lemma}\label{lem:extension-criterion}
For any \(\bar m\in M\models T_0^{\mathrm{col}}\), the formula
\(\operatorname{Ext}_\varphi(\bar m)\) holds in \(M\) if and only if
\(\varphi(\bar x,\bar z;\bar m)\) is realized in some
\(L_{\mathrm{gc}}\)-extension \(N\supseteq M\) satisfying
\(T_0^{\mathrm{col}}\).
\end{lemma}

\begin{proof}
Suppose first that \(N\models T_0^{\mathrm{col}}\) is an
\(L_{\mathrm{gc}}\)-extension of \(M\), and that
\((\bar a,\bar b)\), where \(\bar b=(b_1,\ldots,b_n)\in G(N)^n\), realizes
\(\varphi(\bar x,\bar z;\bar m)\) in \(N\).  Since \(T_0\) eliminates
quantifiers, $
 M|_{L_0}\preccurlyeq N|_{L_0}$. 
Set
\[
 J=\{r\in[n]:b_r\in G(M)\},\qquad S=[n]\setminus J,
\]
set \(c_r=b_r\) for \(r\in J\), and put
\(\bar c_J=(c_r)_{r\in J}\).  Since \(N\) extends \(M\) as an
\(L_{\mathrm{gc}}\)-structure and
\(N\models C_{\eta(r)}(b_r)\), we have
\(M\models C_{\eta(r)}(c_r)\) for every \(r\in J\).  For every
\(s\in S\), one has \(b_s\notin G(M)\).  As $N \models T_0^{\mathrm{col}}$, the predicates \(C_i^N\) give a total good coloring of \(G(N)\) assigning color \(\eta(s)\) to \(b_s\) for every \(s\in S\).  
\Cref{lem:formula-test}  gives $
 N|_{L_0}\models\operatorname{Good}_{\eta_S}(\bar b_S)$.  Since
\(N|_{L_0}\models\psi_0(\bar a,\bar b;\bar m)\) and \(b_r=c_r\) for
\(r\in J\),
it follows that $
 N|_{L_0}\models
 \beta_J(\bar a,\bar b;\bar m,\bar c_J)$. 
Together with \(b_s\notin G(M)\) for \(s\in S\), the tuple
\((\bar a,\bar b)\) realizes the partial \(L_0(M)\)-type (3.4), with
\(\beta=\beta_J\), in the elementary extension \(N|_{L_0}\) of
\(M|_{L_0}\). By \Cref{lem:definable-avoidance}, $
 M\models\operatorname{Av}_{\beta_J,S}(\bar m,\bar c_J)$. 
Thus \(J\) and \(\bar c_J\) witness the \(J\)-indexed disjunct of
\(\operatorname{Ext}_\varphi(\bar m)\).

Conversely, suppose that
\(M\models\operatorname{Ext}_\varphi(\bar m)\).  Choose
\(J\subseteq[n]\), put \(S=[n]\setminus J\), and choose
\(\bar c_J=(c_r)_{r\in J}\in M^{|J|}\) such that
\[
 M\models
 \bigwedge_{r\in J}C_{\eta(r)}(c_r)
 \wedge\operatorname{Av}_{\beta_J,S}(\bar m,\bar c_J).
\]
Since every color predicate $C_i$ is supported on \(G\), each \(c_r\) belongs to
\(G(M)\), so \(\bar c_J\in G(M)^{|J|}\).
By \Cref{lem:definable-avoidance}, there are an elementary
\(L_0\)-extension $
 N_0\succcurlyeq_{L_0}M|_{L_0}$ 
and tuples \(\bar a\in N_0\) and
\(\bar b=(b_1,\ldots,b_n)\in G(N_0)^n\) such that $
 N_0\models\beta_J(\bar a,\bar b;\bar m,\bar c_J)$ 
and \(b_s\notin G(M)\) for every \(s\in S\).  In particular,
\(b_r=c_r\) for every \(r\in J\). Let
\[
 X_{\mathrm{new}}=\bigcup_{s\in S}\Gm b_s\subseteq G(N_0),
\]
with \(X_{\mathrm{new}}=\emptyset\) if \(S=\emptyset\).  This set is
\(\Gm\)-invariant and disjoint from \(G(M)\).  Indeed, \(G(M)\) is
\(\Gm\)-invariant because each \(\gamma\in\Gm\) acts by a parameter-free
\(L_0\)-definable bijection.  Thus, if \(\gamma b_s\in G(M)\) for some
\(\gamma\in\Gm\) and \(s\in S\), then
\(b_s=\gamma^{-1}(\gamma b_s)\in G(M)\), a contradiction.

Suppose \(S\ne\emptyset\).  The definition of \(\beta_J\) gives $
 N_0\models\operatorname{Good}_{\eta_S}(\bar b_S)$. 
By \Cref{lem:formula-test}, there is a total good coloring
\(d:G(N_0)\to[4]\) such that
\(d(b_s)=\eta(s)\) for every \(s\in S\).  Its restriction to
\(X_{\mathrm{new}}\) is good because \(X_{\mathrm{new}}\) is
\(\Gm\)-invariant; put
\(c_{\mathrm{new}}=d|_{X_{\mathrm{new}}}\).  If \(S=\emptyset\), let
\(c_{\mathrm{new}}\) be the unique coloring of the empty set
\(X_{\mathrm{new}}\).

Let \(c_M:G(M)\to[4]\) be the good coloring defined by the predicates of
\(M\).  The sets \(G(M)\) and \(X_{\mathrm{new}}\) are disjoint
\(\Gm\)-invariant subsets of the free \(\Gm\)-set \(G(N_0)\).  By
\Cref{lem:gluing}, the colorings \(c_M\) and \(c_{\mathrm{new}}\) extend to
a total good coloring
\(c:G(N_0)\to[4]\).  Expand \(N_0\) to an \(L_{\mathrm{gc}}\)-structure
\(N\) by declaring, for \(i\in[4]\), $
 C_i^N(v) \ \Leftrightarrow  \ 
 v\in G(N_0) \land \ c(v)=i$. 
Then \(N\models T_0^{\mathrm{col}}\), and \(N\) extends \(M\) as an
\(L_{\mathrm{gc}}\)-structure.  The definition of \(\beta_J\) also gives
\(N_0\models\psi_0(\bar a,\bar b;\bar m)\).  For \(r\in J\), the element
\(b_r=c_r\) has color \(\eta(r)\) because
\(M\models C_{\eta(r)}(c_r)\) and \(c\) extends \(c_M\).  If \(s\in S\),
then \(c(b_s)=c_{\mathrm{new}}(b_s)=d(b_s)=\eta(s)\).
Therefore $
 N\models\varphi(\bar a,\bar b;\bar m)$, 
as required.
\end{proof}

\begin{proposition}\label{prop:model-companion}
The existentially closed models of \(T_0^{\mathrm{col}}\) are axiomatized by
\(T_0^{\mathrm{col}}\) and 
\[
 \forall\bar y\left(
   \operatorname{Ext}_\varphi(\bar y)
   \longrightarrow
   \exists\bar x\,\exists\bar z\,
      \varphi(\bar x,\bar z;\bar y)
 \right)                                                    \tag{3.6}
\]
for all normalized formulas (3.3).  Their theory is the model companion
\(T_0^{\mathrm{gc}}\).
\end{proposition}

\begin{proof}
\Cref{lem:existential-normal-form,lem:extension-criterion} imply that a model $M \models T_0^{\mathrm{col}}$ is existentially closed exactly when it satisfies  the axiom
scheme (3.6) (when \(n=0\), the extension predicate is
\(\exists\bar x\,\psi_0(\bar x;\bar y)\), so the corresponding purely
\(L_0\) axiom is tautological). And \(T_0^{\mathrm{col}}\) is inductive. Indeed, let
\((M_i)_{i\in I}\) be a chain of models of \(T_0^{\mathrm{col}}\), and put
\(M=\bigcup_{i\in I}M_i\).  Since \(T_0\) eliminates quantifiers, the
inclusions between the \(L_0\)-reducts of the \(M_i\) are elementary, so
\(M|_{L_0}\models T_0\).  Every finite tuple from \(M\) lies in some
\(M_i\), where the color predicates axioms hold.  Hence they
also hold in \(M\), and \(M\models T_0^{\mathrm{col}}\). Thus 
$T_0^{\mathrm{col}}$ together with the axiom
scheme (3.6)  is the model companion of \(T_0^{\mathrm{col}}\) (see e.g.~\cite[Theorem~8.3.6]{Hodges}).
\end{proof}

\subsection{Types and algebraic closure}

\begin{lemma}\label{lem:independent-gluing}
Let \(N\models T_0\), and let \(C\subseteq A,B\subseteq N\) be
\(\acl_0\)-closed, with \(A\ind^0_C B\).  Suppose that \(G(A)\) and
\(G(B)\) carry good colorings which agree on \(G(C)\).  Then these colorings
combine and extend to a total good coloring of \(G(N)\).
\end{lemma}

\begin{proof}
The assumption implies \(A\cap B=C\): an element of the intersection
is independent from itself over \(C\), hence belongs to \(\acl_0(C)=C\).
If \(w\in\Gm\) and \(u\in G(A)\), then the map
\(u\mapsto wu\) is parameter-free definable, so every coordinate of the tuple \(wu\)
lies in \(\operatorname{dcl}_0(A)\subseteq\acl_0(A)=A\).  Thus
\(wu\in G(A)\), and the same argument for \(w^{-1}\) gives
\(wG(A)=G(A)\).  Likewise, \(wG(B)=G(B)\).  Hence these sets are
\(\Gm\)-invariant, and
\(G(A)\cap G(B)=G(A\cap B)=G(C)\).  Apply \Cref{lem:gluing}.
\end{proof}

\begin{proposition}\label{prop:relative-homogeneity}
Let \(M_1,M_2\) be models of the same completion of
\(T_0^{\mathrm{gc}}\).  Every color-preserving \(L_0\)-isomorphism between
small \(\acl_0\)-closed subsets of \(M_1\) and \(M_2\) is partial
\(L_{\mathrm{gc}}\)-elementary.
\end{proposition}

\begin{proof}
Let \(f:A\to B\) be such an isomorphism.  Since the completion is
model-complete, every formula is equivalent modulo it to an existential
formula.  It is therefore enough to show that \(f\) and \(f^{-1}\) preserve
existential formulas.

Suppose
\(M_1\models\exists\bar x\,\theta(\bar x,\bar a)\), where
\(\bar a\in A\) and \(\theta\) is  a quantifier-free \(L_{\mathrm{gc}}\)-formula, and choose a witness
\(\bar c\).  Put \(D=\acl_0(A\bar c)\). By  quantifier elimination in $T_0$, \(f\) is  \(L_0\)-elementary.  Stability and extension give, in an
elementary \(L_0\)-extension \(N\) of \(M_2\), an \(\acl_0\)-closed set
\(D'\) and an \(L_0\)-elementary isomorphism $
 g:D \to D'$ 
extending \(f\), with \(D'\ind^0_BM_2\).  Thus \(D'\cap M_2=B\).
(Here we applied nonforking extension to the transported type $f_{\ast}(\tp(D/A))$ over \(B\) of
an enumeration of \(D\); the last equality follows from anti-reflexivity of $\ind^0$ and
\(\acl_0(B)=B\)).

Let \(c^{M_i}:G(M_i)\to[4]\) denote the coloring defined by the predicates
of \(M_i\), for \(i=1,2\).  Since \(D\) is \(\acl_0\)-closed, \(G(D)\) is
\(\Gm\)-invariant.  Transport its coloring to \(G(D')\) by setting
\(c_{D'}(v)=c^{M_1}(g^{-1}(v))\).  The map \(g\) sends \(G(D)\) bijectively
onto \(G(D')\) and commutes with every parameter-free word map, so
\(c_{D'}\) is good.  Since \(g|_A=f\) and \(f\) is color-preserving,
\(c_{D'}\) agrees with \(c^{M_2}\) on \(G(B)\). 
\Cref{lem:independent-gluing} applied to \(D'\ind^0_BM_2\) gives a total
good coloring \(\widehat c:G(N)\to[4]\) extending \(c_{D'}\) and
\(c^{M_2}\).  Let \(\widetilde N\) be the \(L_{\mathrm{gc}}\)-expansion of
\(N\) defined by $
 C_i^{\widetilde N}(v) \Leftrightarrow 
 v\in G(N)\ \land \ \widehat c(v)=i$. 
The goodness of \(\widehat c\) gives
\(\widetilde N\models T_0^{\mathrm{col}}\), and the fact that
\(\widehat c\) extends \(c^{M_2}\) makes
\(M_2\subseteq\widetilde N\) an \(L_{\mathrm{gc}}\)-substructure.

Because \(G\) is \(\emptyset\)-definable, \(g\) preserves membership in
\(G\).  By the definitions of \(c_{D'}\) and \(\widehat c\), it preserves
every color predicate on \(G(D)\); off \(G\), all color predicates are false.
 Since \(D\) and \(D'\) are \(\acl_0\)-closed, they are
\(\operatorname{dcl}_0\)-closed \(L_0\)-substructures.  Thus \(g\) is an
\(L_{\mathrm{gc}}\)-isomorphism between the substructures induced on \(D\)
in \(M_1\) and on \(D'\) in \(\widetilde N\).  Since
\(\bar a,\bar c\subseteq D\) and \(\theta\) is quantifier-free, it follows
that $
 \widetilde N\models\theta(g(\bar c),f(\bar a))$. 
Now existential closedness of \(M_2\) among models of
\(T_0^{\mathrm{col}}\) gives
\(M_2\models\exists\bar x\,\theta(\bar x,f(\bar a))\).  The same argument
for \(f^{-1}\) gives the converse implication.  Thus \(f\) preserves all
existential formulas and is partial elementary.
\end{proof}

\begin{corollary}\label{cor:acl}
In every completion of \(T_0^{\mathrm{gc}}\), $
 \acl_{\mathrm{gc}}(A)=\acl_0(A)$ 
for every small set \(A\).
\end{corollary}

\begin{proof}
Only the inclusion from left to right requires proof.  Fix a completion and
let \(\U\) be its monster.  Suppose \(b\notin\acl_0(A)\), and put $
 C=\acl_0(A)$, $D=\acl_0(Cb)$. 
In a sufficiently saturated elementary extension \(N\) of the \(L_0\)-reduct of \(\U\),
stable extension gives, recursively for \(n<\omega\), \(L_0\)-elementary
isomorphisms \(g_n:D\to D_n\) over \(C\), with
\[
 D_n\ind^0_C\left(\U\cup\bigcup_{m<n}D_m\right).
\]
Hence \(D_n\cap\U=C\) and \(D_n\cap D_m=C\) for \(n\ne m\).

Transport the coloring of \(G(D)\) to every \(G(D_n)\).  These colorings
and the ambient coloring on \(G(\U)\) agree on all pairwise intersections,
so \Cref{lem:gluing} extends them to a total good coloring of \(G(N)\).  Put this
colored structure in a sufficiently saturated existentially closed extension
\(W\).  Model-completeness and quantifier elimination give $
 \U\preccurlyeq_{L_{\mathrm{gc}}}W$,
$ N\preccurlyeq_{L_0}W|_{L_0}$. 
Each \(g_n\) is now a color-preserving \(L_0\)-isomorphism between
\(\acl_0\)-closed sets, so \Cref{prop:relative-homogeneity} makes it partial
\(L_{\mathrm{gc}}\)-elementary.  Thus the elements \(g_n(b)\), \(n<\omega\),
are distinct realizations of \(\tp_{L_{\mathrm{gc}}}(b/A)\): if
\(g_n(b)=g_m(b)\) for \(n\ne m\), that element belongs to
\(D_n\cap D_m=C\), contrary to \(b\notin C\).  Therefore
\(b\notin\acl_{\mathrm{gc}}(A)\).
\end{proof}

\subsection{Simplicity}
We prove that the independence theorem over models holds in any completion of \(T_0^{\mathrm{gc}}\): 
\begin{lemma}\label{lem:independent-amalgamation}
Let \(T^\dagger\) be a completion of \(T_0^{\mathrm{gc}}\), let \(\U\) be
its monster, and let \(M\preccurlyeq\U\) be small.  Suppose that
\(M\subseteq A,B\subseteq\U\) and that finite tuples \(d_0,d_1\) satisfy
\[
 A\ind^0_MB,\qquad d_0\ind^0_MA,\qquad d_1\ind^0_MB,
 \qquad d_0\equiv_M^{L_{\mathrm{gc}}}d_1.
\]
Then there is \(d\in\U\) such that $
 d\equiv_A^{L_{\mathrm{gc}}}d_0$, 
 $d\equiv_B^{L_{\mathrm{gc}}}d_1$, and 
 $d\ind^0_MAB$. 
\end{lemma}

\begin{proof}
Replacing \(A\) and \(B\) by their \(\acl_0\)-closures preserves the
hypotheses and strengthens the conclusion, so assume that they are
\(\acl_0\)-closed.  Put \(D_i=\acl_0(Md_i)\).  Choose
\(\sigma\in\operatorname{Aut}_{L_{\mathrm{gc}}}(\U/M)\) with
\(\sigma(d_0)=d_1\), and let \(s=\sigma|_{D_0}:D_0\to D_1\).

In a sufficiently saturated elementary extension
\(N\succcurlyeq_{L_0}\U|_{L_0}\), choose $D' \equiv^{L_0}_M D_0$ with \(D'\ind^0_M\U\), and chose  an \(L_0\)-elementary map $
 t:D_0\longrightarrow D'$ 
fixing \(M\) pointwise,  and put \(d'=t(d_0)\).  The independence
hypotheses give \(D_0\cap A=M\) and \(D_1\cap B=M\); moreover,
\(D'\cap A=D'\cap B=M\).  Stationarity over the
model \(M\) therefore gives \(L_0\)-elementary maps
\[
\begin{aligned}
 f_A:\acl_0(AD_0)&\longrightarrow X:=\acl_0(AD'),\\
 f_B:\acl_0(BD_1)&\longrightarrow Y:=\acl_0(BD'),
\end{aligned}
\]
where \(f_A\) fixes \(A\) and extends \(t\), while \(f_B\) fixes \(B\) and
extends \(t\circ s^{-1}\) (this is stationarity because
\(D_0\ind^0_MA\), \(D_1\ind^0_MB\), and \(D'\ind^0_M\U\)).

We also have
\[
 X\cap\U=A,\qquad Y\cap\U=B,\qquad X\cap Y=D'.             \tag{3.7}
\]
Indeed, \(D'\ind^0_M\U\) gives
\(X\ind^0_A\U\) and \(Y\ind^0_B\U\), yielding the first two equalities.
For the third, \(A\ind^0_MB\) and \(D'\ind^0_MAB\) give
\(AD'\ind^0_{D'}BD'\),  hence
\(X\ind^0_{D'}Y\).

Push the ambient colorings on the domains of \(f_A\) and \(f_B\) forward to
\(G(X)\) and \(G(Y)\).  The transported colorings are good because the maps
commute with the $\emptyset$-definable \(\Gm\)-action.  They agree with the ambient
coloring on \(G(A)\) and \(G(B)\), since the maps fix \(A\) and \(B\).  They agree with each
other on \(G(D')\): every element there is \(t(x)\) for some
\(x\in G(D_0)\), and $
 f_A(x)=t(x)=f_B(s(x))$ for all $x\in G(D_0)$, 
and \(s\) is color-preserving. The pairwise intersections are therefore $
 G(\U)\cap G(X)=G(A)$, $G(\U)\cap G(Y)=G(B)$, $
 G(X)\cap G(Y)=G(D')$. 
 Thus (3.7) and \Cref{lem:gluing}, applied to
the \(\Gm\)-invariant sets \(G(\U),G(X),G(Y)\), give a total good coloring
of \(G(N)\) extending all three colorings.  Let
\(\widetilde N\models T_0^{\mathrm{col}}\) be the
resulting expansion of \(N\).

Since \(T_0^{\mathrm{gc}}\) is the model companion of
\(T_0^{\mathrm{col}}\), extend \(\widetilde N\) to a model of
\(T_0^{\mathrm{gc}}\), and pass to a sufficiently
saturated elementary extension \(W\).  Model-completeness and quantifier
elimination give $
 \U\preccurlyeq_{L_{\mathrm{gc}}}W \models T^\dagger$ and $N\preccurlyeq_{L_0}W|_{L_0}$. 
Consequently the domains and ranges of \(f_A,f_B\) remain
\(\acl_0\)-closed in $W$, and the transported colorings make both maps
color-preserving.  Then, by proposition \ref{prop:relative-homogeneity}, they are 
\(L_{\mathrm{gc}}\)-elementary.  Hence \(d'=f_A(d_0)=f_B(d_1)\) realizes
both required colored types, while \(D'\ind^0_M\U\) gives
\(d'\ind^0_MAB\).  Finally, since \(\U\preccurlyeq W\) and \(\U\) is
saturated, choose \(d\in\U\) realizing
\(\tp_{L_{\mathrm{gc}}}(d'/AB)\). \end{proof}

\begin{proposition}\label{prop:simplicity}
Every completion of \(T_0^{\mathrm{gc}}\) is simple, and its nonforking
relation coincides with nonforking in $T_0$.
\end{proposition}

\begin{proof}
Fix a completion \(T^\dagger\) and its monster \(\U\).  We show that \(\ind^0\) satisfies the Kim--Pillay axioms  \cite[Definition~4.1 and Theorem~4.2]{KimPillay}.  Since every
\(L_{\mathrm{gc}}\)-automorphism is an \(L_0\)-automorphism, invariance
follows from invariance of nonforking in \(T_0\).  Stability of \(T_0\) also
gives local character, finite character, symmetry, and transitivity.  
\Cref{lem:independent-amalgamation} is precisely the independence theorem
over models for \(\ind^0\). 
Thus only extension remains to be verified.

Let \(a\) be a finite tuple and let
\(C\subseteq B\subseteq\U\).  Put \(E=\acl_0(C)\) and
\(D=\acl_0(Ea)\).  In an elementary \(L_0\)-extension \(N\) of
\(\U|_{L_0}\), choose an \(L_0\)-elementary map
\(t:D\to D'\) fixing \(E\) pointwise, with \(D'\ind^0_E\U\).  Transport the coloring
of \(G(D)\) to \(G(D')\).  It is good and agrees with the ambient coloring
on \(G(E)\), so \Cref{lem:independent-gluing} extends these two colorings to
a total good coloring of \(G(N)\).  Let
\(\widetilde N\models T_0^{\mathrm{col}}\) be the resulting colored
expansion. Extend \(\widetilde N\) to a model \(V\models T_0^{\mathrm{gc}}\), then  
\(\U\preccurlyeq_{L_{\mathrm{gc}}}V\); in particular,
\(V\models T^\dagger\).  Pass to a sufficiently saturated elementary
extension \(W\succcurlyeq_{L_{\mathrm{gc}}} V\).  Quantifier elimination for \(T_0\) gives
\(N\preccurlyeq_{L_0}W|_{L_0}\), so \(D\) and \(D'\) remain
\(\acl_0\)-closed in \(W\).  The map \(t\) is color-preserving, hence
\(L_{\mathrm{gc}}\)-elementary by \Cref{prop:relative-homogeneity}.  Thus
\(t(a)\equiv_C^{L_{\mathrm{gc}}}a\).  Moreover,
\(D'\ind^0_E\U\) gives \(t(a)\ind^0_EB\), and, since
\(E=\acl_0(C)\), this implies \(t(a)\ind^0_CB\).  As
\(\U\preccurlyeq W\) and \(\U\) is saturated, choose \(a'\in\U\) realizing
\(\tp_{L_{\mathrm{gc}}}(t(a)/B)\).  Then
\(a'\equiv_C^{L_{\mathrm{gc}}}a\), and equality of the complete
\(L_0(B)\)-types gives \(a'\ind^0_CB\).  
\end{proof}

\begin{proof}[Proof of \Cref{thm:generic-color}]
The theorem follows from
\Cref{lem:measure-obstruction,prop:model-companion,prop:relative-homogeneity,cor:acl,prop:simplicity}.
\end{proof}

\section{The free-group counterexample}
\label{sec:free-group}

Let
\[
 F=F_{12}=\langle a_{ij}:i\in[4],\ j\in[3]\rangle,
 \qquad F^+=F*\langle c\rangle=F_{13}.
\]
Let \(L_{\mathrm{grp}}\) be the group language,
\(L_a=L_{\mathrm{grp}}\cup\{a_{ij}:i\in[4],j\in[3]\}\), and $
  T_a=\operatorname{Th}_{L_a}(F,(a_{ij})_{i,j})$. 
Its reduct to \(L_{\mathrm{grp}}\) is the common theory of nonabelian free
groups.  Sela proved stability \cite{SelaStability}, and Sklinos proved nfcp
\cite{SklinosNFCP}; naming constants preserves both
properties.  Moreover, the free-factor inclusion $
 F\preccurlyeq F^+$  
is elementary in the pure group language; see
\cite[Theorem B(i)]{PillayFree}.  Since the inclusion fixes every
\(a_{ij}\), it remains elementary in \(L_a\).

Let \(T_0\) be  a relational
Morleyization of \(T_a\) in the language \(L_0\).  Thus \(T_0\)
is complete, stable, nfcp, and has quantifier elimination.  We continue to write \(F\) and \(F^+\) for
their \(L_0\)-expansions, the inclusion $F\preccurlyeq F^+$  remains $L_0$-elementary. 
The group is the home sort and the \(a_{ij}\) are named, so the hypotheses of
\Cref{thm:generic-color} hold.  Fix a completion
\(T^\sharp\) of \(T_0^{\mathrm{gc}}\), and let
\(\U\models T^\sharp\) be a monster model with saturation cardinal \(\bar\kappa\) larger than
\(2^{\aleph_0}\).  By \Cref{thm:generic-color},
\(T^\sharp\) is simple.

\begin{proposition}
\label{prop:bounded-orbit}
There exist an elementary extension
\(W\succcurlyeq_{L_{\mathrm{gc}}}\U\), an element
\(c\in W\), and a (non-definable) subgroup \(H\leq\U\) such that
\[
 [\U:H]\leq2^{\aleph_0}
 \quad\text{and}\quad
 c\equiv_{\U}^{L_{\mathrm{gc}}}hc\quad\text{for every }h\in H. \tag{4.2}
\]
\end{proposition}

\begin{proof}
\emph{An ultrapower construction.}
Since \(F\) is finitely generated, it has only countably many finite-index
normal subgroups (each is the kernel of an epimorphism onto a finite group
\(R\), and for fixed
\(R\) there are at most \(|R|^{12}\) homomorphisms \(F\to R\)).  Enumerate them as
\((N_n)_{n<\omega}\), put \(R_n=F/N_n\), and let \(q_n:F\to R_n\) be the quotient
map.  Residual finiteness of \(F\) gives
\(\bigcap_{n<\omega}N_n=\{1\}\); equivalently, the maps \(q_n\) separate points:
if \(f\neq g \in F\), then \(q_n(f)\neq q_n(g)\) for some \(n\).  Extend \(q_n\) by
\(\widetilde q_n(c)=1\) to a homomorphism $
 \widetilde q_n:F^+ \to R_n$. 
For \(h\in F\), define \(\nu_h:F^+\to F^+\) by
\(\nu_h|_F=\mathrm{id}_F\) and \(\nu_h(c)=hc\).  The universal property of the
free product gives a unique such homomorphism, with inverse \(\nu_{h^{-1}}\), so
\(\nu_h\in\operatorname{Aut}(F^+/F)\).  It is an \(L_0\)-automorphism: it
fixes the named \(a_{ij}\), and all other \(L_0\)-relations are definitional
Morley predicates.

We shall use the elementary identity
\[
 q_n(h)=1\quad\Longrightarrow\quad
 \widetilde q_n\bigl(\nu_h(y)\bigr)=\widetilde q_n(y)
 \qquad(n<\omega,\ h\in F,\ y\in F^+).                     \tag{*}
\]
If \(q_n(h)=1\), then \(\widetilde q_n\circ\nu_h\) and \(\widetilde q_n\)
agree on \(F\), since \(\nu_h|_F=\mathrm{id}_F\), while
\(\widetilde q_n(\nu_h(c))=\widetilde q_n(hc)
=q_n(h)\widetilde q_n(c)=1=\widetilde q_n(c)\).  They therefore agree on
\(F^+=F*\langle c\rangle\), proving (*).

Since \(\U|_{L_0}\equiv F\), the ultrapower embedding lemma gives an
ultrafilter \(\mathcal D\) on a set \(I\) and an \(L_0\)-elementary embedding $
 e:\U|_{L_0} \to P^*:=F^I/\mathcal D$. Write
\[
 \Gamma=\langle a_{ij}:i\in[4],\ j\in[3]\rangle\leq\U.
\]
Because the \(a_{ij}\) are named, \(e\) identifies \(\Gamma\) pointwise with
the diagonal copy of \(F\) in \(P^*\).

Put
\[
 Q^*:=(F^+)^I/\mathcal D,\qquad
 c=[i\mapsto c]_{\mathcal D}\in Q^*.
\]
The elementary  inclusion \(F\preccurlyeq_{L_0}F^+\) and
\L o\'s's theorem give \(P^*\preccurlyeq_{L_0}Q^*\).  Identifying \(\U\)
with \(e(\U)\) and \(P^*\) with its image under the coordinatewise inclusion
in \(Q^*\), we have
\[
 \U|_{L_0}\preccurlyeq P^*\preccurlyeq Q^*.                 \tag{4.3}
\]

The quotient maps also extend coordinatewise.  For
\(x=[x_i]_{\mathcal D}\in P^*\), define \(q_n^*(x)\) to be the unique
\(r\in R_n\) such that $
 \{i\in I:q_n(x_i)=r\}\in\mathcal D$; 
such a unique \(r\) exists because \(R_n\) is finite.  Define
\(\widetilde q_n^*:Q^*\to R_n\) in the same way from
\(\widetilde q_n\).  These are well-defined homomorphisms, \(\widetilde q_n^*|_{P^*}=q_n^*\) and
\[
 q_n^*(h)=1
 \quad\Longleftrightarrow\quad
 \{i\in I:q_n(h_i)=1\}\in\mathcal D
\]
for \(h=[h_i]_{\mathcal D}\in P^*\).

\smallskip
\emph{The finite-quotient kernel.}
Let 
\[
 \rho_n=q_n^*|_{\U},\qquad
 \rho=(\rho_n)_{n<\omega}:\U\longrightarrow
       \prod_{n<\omega}R_n,
 \qquad H=\ker(\rho).                                      \tag{4.4}
\]
 Under the displayed identification,
\(q_n^*|_\Gamma=q_n\), and hence \(\rho_n|_\Gamma=q_n\).  The maps \(q_n\)
separate points of \(F\), so
\(\rho|_\Gamma\) is injective and therefore \(H\cap\Gamma=\{1\}\).  By the
first isomorphism theorem 
\[
 [\U:H]\leq \left| \prod_{n<\omega}R_n \right| \leq 2^{\aleph_0}<\bar\kappa.                                 \tag{4.5}
\]

For \(h\in H\), choose a representative
\(h=[h_i]_{\mathcal D}\), and define
\[
 \alpha_h([y_i]_{\mathcal D})
   :=[\nu_{h_i}(y_i)]_{\mathcal D}\qquad(y_i\in F^+).
\]
This is well defined: if \(h=[h_i]=[h_i']\) and \(y=[y_i]=[y_i']\), then
\(h_i=h_i'\) and \(y_i=y_i'\) simultaneously on a
\(\mathcal D\)-large set, so the two output sequences agree there.  Since
every \(\nu_{h_i}\) is an \(L_0\)-automorphism of \(F^+\) fixing \(F\)
pointwise, \L o\'s's theorem
shows that \(\alpha_h\) is an \(L_0\)-automorphism of \(Q^*\) fixing \(P^*\)
pointwise; its inverse is defined coordinatewise by \(\nu_{h_i^{-1}}\).
Moreover,
\(\alpha_h(c)=[h_ic]_{\mathcal D}=hc\).  Thus
\[
 \alpha_h\in\operatorname{Aut}_{L_0}(Q^*/P^*),
 \qquad \alpha_h(c)=hc.                                     \tag{4.6}
\]

For each fixed \(n\), the equality \(q_n^*(h)=1\) means that $
 I_n(h):=\{i\in I:q_n(h_i)=1\}\in\mathcal D$. 
On \(I_n(h)\), the identity (*) gives
\(\widetilde q_n(\nu_{h_i}(y_i))=\widetilde q_n(y_i)\).
Therefore
\(\widetilde q_n^*(\alpha_h(y))=\widetilde q_n^*(y)\) for every \(y\in Q^*\).
 Thus
\(\widetilde q_n^*\circ\alpha_h=\widetilde q_n^*\) for every \(n\).
Consequently every element $\sigma$ of $
 K=\langle\alpha_h:h\in H\rangle$ 
fixes \(P^*\) pointwise and satisfies \(\widetilde q_n^*\circ \sigma =\widetilde q_n^*\) for every \(n\).

\smallskip
\emph{A \(K\)-invariant good coloring.}
Let \(Y\) be the set of \(K\)-orbits on \(Q^*\), and write \([y]_K\) for
the orbit of \(y\).  If \(\sigma\in K\), \(\gamma\in\Gamma\), and \(y\in Q^*\),
then \(\sigma\) is a group automorphism fixing \(\Gamma\) pointwise, so
\(\sigma(\gamma y)=\sigma(\gamma)\sigma(y)=\gamma\sigma(y)\).  Thus left
multiplication by \(\Gamma\) descends to an action on \(Y\): explicitly, $
  \gamma\mathbin{\cdot}[y]_K:=[\gamma y]_K$.
The preceding commutation identity shows that this definition is independent
of the representative \(y\) of the \(K\)-orbit.

This action is
free.  Indeed, if
\(\gamma[y]_K=[y]_K\), then
\(\gamma y=\sigma(y)\) for some \(\sigma\in K\).  Applying
\(\widetilde q_n^*\) gives $
 q_n(\gamma)\widetilde q_n^*(y)
 =\widetilde q_n^*(\gamma y)
 =\widetilde q_n^*(\sigma(y))
 =\widetilde q_n^*(y)$, 
because every \(\sigma\in K\) preserves \(\widetilde q_n^*\).  Cancelling
gives \(q_n(\gamma)=1\) for every \(n\); since the \(q_n\) separate points of
\(\Gamma\cong F\), we get \(\gamma=1\).

Let \(\pi:Q^*\to Y\) be the orbit projection.  It is injective on \(\U\),
because \(K\) fixes \(P^*\) pointwise.  Its image \(\pi(\U)\) is
\(\Gamma\)-invariant, since \(\Gamma\leq\U\) and
\(\gamma\cdot\pi(u)=\pi(\gamma u)\in\pi(\U)\) for
\(\gamma\in\Gamma\), \(u\in\U\).  Define \(\delta:\pi(\U)\to[4]\) by
\(\delta(\pi(u))=i\) if and only if \(\U \models C_i(u)\); injectivity makes this well
defined.  It is good: given \(u\in\U\) and \(i\in[4]\), choose \(j\in[3]\)
with \(\U \models C_i(a_{ij}u)\); then \(\Gamma\)-equivariance gives
\(\delta(a_{ij}\cdot\pi(u))=\delta(\pi(a_{ij}u))=i\).
By \Cref{lem:gluing}, extend \(\delta\) to a total good coloring
\(\widehat\delta:Y\to[4]\). For $y \in Q^*$, let $\textrm{col}(y):=\widehat\delta(\pi(y))$, and define \(C_i(y)\) on \(Q^*\) by
\(\textrm{col}(y)=i\).  Then $\textrm{col}$ is good by $\Gamma$-equivariance of $\pi$ (given \(y\in Q^*\) and \(i\in[4]\),
the goodness of \(\widehat\delta\) yields some \(j\in[3]\) such that $
  \widehat\delta(a_{ij}\cdot\pi(y))=i$, and $
  a_{ij}\cdot\pi(y)=\pi(a_{ij}y)$, hence \(\textrm{col}(a_{ij}y)=i\)), extends
the coloring of \(\U\) (if \(u\in\U\), then $\textrm{col}(u)=\widehat\delta(\pi(u))
      =\delta(\pi(u))$, 
which, by the definition of \(\delta\), is the original color of \(u\)), and is \(K\)-invariant (because
\(\pi(\sigma(y))=\pi(y)\) for \(\sigma\in K\)).

Thus \((Q^*, \textrm{col})\) is a
model of \(T_0^{\mathrm{col}}\) containing \(\U\) as an
\(L_{\mathrm{gc}}\)-substructure.  By \Cref{prop:model-companion}, embed
\(Q^*\) in an existentially closed
\(T_0^{\mathrm{col}}\)-model
\(V\models T_0^{\mathrm{gc}}\); identify \(Q^*\) with its image.
Since
\(\U\models T_0^{\mathrm{gc}}\), model-completeness from 
\Cref{prop:model-companion} gives
\(\U\preccurlyeq_{L_{\mathrm{gc}}}V\); in particular,
\(V\models T^\sharp\).  Take a sufficiently saturated
\(W\succcurlyeq_{L_{\mathrm{gc}}}V\).  By quantifier elimination in \(T_0\) we thus have   
\[
 \U\preccurlyeq_{L_{\mathrm{gc}}} W,
 \qquad Q^*|_{L_0}\preccurlyeq W|_{L_0}.                    \tag{4.7}
\]

\smallskip
\emph{Stabilization of the type.}
Fix \(h\in H\) and a finite set \(A\subseteq\U\).  Since
\(Q^*|_{L_0}\preccurlyeq W|_{L_0}\),
\[
 \acl_0^W(A,c)=\acl_0^{Q^*}(A,c),\qquad
 \acl_0^W(A,hc)=\acl_0^{Q^*}(A,hc).
\]
By (4.6), \(\alpha_h\in\operatorname{Aut}_{L_0}(Q^*/P^*)\) sends \(c\) to
\(hc\); since \(A\subseteq\U\subseteq P^*\), it fixes \(A\) pointwise and carries
\(\acl_0^{Q^*}(A,c)\) onto \(\acl_0^{Q^*}(A,hc)\).  As \(\alpha_h\in K\)
and the coloring $\textrm{col}$ of $Q^*$ is \(K\)-invariant, this restriction preserves every color $C_i$.  \Cref{prop:relative-homogeneity} therefore gives
\(c\equiv_A^{L_{\mathrm{gc}}}hc\).  Every formula over \(\U\) has only
finitely many parameters, so
\[
 c\equiv_{\U}^{L_{\mathrm{gc}}}hc\qquad(h\in H).             \tag{4.8}
\]
Together with (4.5), this proves (4.2).
\end{proof}

\begin{proof}[Proof of \Cref{thm:simple-main}]
Let $
 L=L_{\mathrm{grp}}\cup\{C_1,C_2,C_3,C_4\}$ and $T=T^\sharp\!\upharpoonright L$. 
Thus \(T\) is a complete expansion of the theory of nonabelian free groups by
four unary predicates (note that \(\U \restriction_{L}\) is a monster model for \(T\)). Let \(c,H,W\) be supplied by \Cref{prop:bounded-orbit}, and let $
 p=\tp_L(c/\U)$. 
Equation (4.8) implies that
\(H\leq\Stab(p)\).  Hence
\[
 |\{g\cdot p:g\in\U\}|=[\U:\Stab(p)]
 \leq[\U:H]\leq2^{\aleph_0}<\bar\kappa.                  \tag{4.9}
\]
Simplicity passes from \(T^\sharp\) to its reduct \(T\).  The proof of
\Cref{lem:measure-obstruction} remains valid in \(L\), so the home-sort group is not definably amenable. 
\end{proof}

\section{Final remarks and future directions}

\begin{remark}\label{rem: SL2 no bdd orbit}
	The earlier four-colored \(G = \mathrm{SL}_2\) construction over an algebraically
closed field  \cite[Section~3]{CHK} does not already give a 
counterexample.  Indeed, for a monster model of a completion of that colored theory,
let \(K\) be the underlying algebraically closed field and put
\(G=\mathrm{SL}_2(K)\).  In the colored expansion the translate-covering axioms destroy
definable amenability, but there is still no global type in the colored
language with bounded orbit.  
Indeed, suppose that \(H\leq G\) and put
\([G:H]<\bar\kappa\), where $\U$ is $\bar\kappa$-saturated.  The kernel of the left action on \(G/H\) is the
normal core
\(N=\bigcap_{x\in G}xHx^{-1}\), so \(N\trianglelefteq G\) and \(N\leq H\).
Since this action embeds \(G/N\) in \(\operatorname{Sym}(G/H)\),
\([G:N]\leq2^{[G:H]+\aleph_0}<\bar\kappa\), so \(N\) also has bounded
index.  For the quotient map
\(\pi:G\to G/Z(G)=\mathrm{PSL}_2(K)\), abstract simplicity says that
\(\pi(N)\) is either trivial or all of \(\mathrm{PSL}_2(K)\).  The first
case would give \(N\leq Z(G)\), contradicting
\([G:N]\geq[G:Z(G)]=\bar\kappa\), since \(Z(G)=\{\pm I\}\).  In the second
case \(NZ(G)=G\), and hence
\(G/N\cong Z(G)/(Z(G)\cap N)\) is abelian.  But
\(G=\mathrm{SL}_2(K)\) is perfect, so \(G/N\) is also perfect and must be
trivial.  Therefore \(N=G\), and then \(H=G\). If \(p\in S_G(\U)\) had bounded orbit, then \(\Stab(p)\) would have bounded
index and hence equal \(G\).  Thus \(p\) would be a  left-invariant
\(\{0,1\}\)-valued Keisler measure, contradicting the covering axioms.
\end{remark}

\begin{problem}
Is there a counterexample to Petrykowski's conjecture of finite SU-rank? The theory of non-abelian free groups is not superstable by \cite{poizat1983groupes}.
\end{problem}

Recall that a definable group $G$  has 
\emph{finitely satisfiable generics} (fsg) if there are a global type
\(q\in S_G(\U)\) and a small model \(M\preccurlyeq\U\) such that every left
translate of \(q\) is finitely satisfiable in \(M\) \cite{hrushovski2008groups}.  In particular, the
orbit of \(q\) is bounded. 
\begin{remark}
The group we constructed in Theorem \ref{thm:simple-main} is not fsg.  First, no color \(C_i\) is
right generic.  Given \(g_1,\ldots,g_m\in\U\), compactness in the pure-group
reduct gives, in an elementary extension, an \(x\) for which the points
\(xg_k^{-1}\) lie in distinct new \(\Gamma\)-orbits.  Indeed, after deleting
repetitions, a finite fragment excludes only finitely many points, to keep the
orbits off \(\U\), and finitely many sets of solutions to
\(x^{-1}\gamma x=d\), with \(1\ne\gamma\in\Gamma\).  In the free group
\(\Gamma\), each latter set is empty or, after choosing one solution \(x_0\),
equal to \(C_\Gamma(\gamma)x_0\), hence is a coset of a cyclic subgroup.
Finitely many such cosets and points cannot cover \(\Gamma\): points are cosets
of the trivial subgroup, while every cyclic subgroup of
\(\Gamma\cong F_{12}\) has infinite index, and Neumann's lemma
says that in any finite cover of a group by cosets, at least one underlying
subgroup has finite index \cite[Lemma~4.1]{NeumannCosets}.
Give all the points \(xg_k^{-1}\) a fixed color different from \(i\).
Since they lie in pairwise distinct \(\Gamma\)-orbits disjoint from
\(G(\U)\), \Cref{lem:gluing} extends these
prescriptions together with the ambient coloring on \(G(\U)\) to a total
good coloring of the group of the elementary extension. Existential closedness then gives such an \(x\) already in \(\U\).  Thus
\(x\) lies outside every \(C_i g_k\), as required.

Now suppose that \(q\) witnessed fsg over a small model \(M\), and let
\(X\in q\).  For each \(g\in\U\), finite satisfiability of \(g\cdot q\) in
\(M\) gives \(m\in M\cap gX\), whence \(g\in mX^{-1}\).  Therefore
\(\U=\bigcup_{m\in M}mX^{-1}\); saturation gives a finite subcover, so \(X\)
is right generic.  But \(q\) contains one of the four colors \(C_i\), a
contradiction.  Hence the group is not fsg. 
\end{remark}

Thus the following problem remains open:
\begin{problem}
	\cite[Problem 3.31]{CGK} Is every definable fsg group in an arbitrary theory definably amenable? 
\end{problem}

\bibliographystyle{plain}
\bibliography{refs.bib}

\begin{thebibliography}{10}

\bibitem{CGK}
Artem Chernikov, Kyle Gannon, and Krzysztof Krupi{\'n}ski.
\newblock Definable convolution and idempotent {K}eisler measures {III}.
  {G}eneric stability, generic transitivity, and revised {N}ewelski's
  conjecture.
\newblock {\em Journal of the London Mathematical Society}, 114(1):e70639,
  2026.

\bibitem{CHK}
Artem Chernikov, Ehud Hrushovski, Alex Kruckman, Krzysztof Krupi{\'n}ski,
  Slavko Moconja, Anand Pillay, and Nicholas Ramsey.
\newblock Invariant measures in simple and in small theories.
\newblock {\em Journal of Mathematical Logic}, 23(2):2250025, 2023.

\bibitem{CS}
Artem Chernikov and Pierre Simon.
\newblock Definably amenable {NIP} groups.
\newblock {\em Journal of the American Mathematical Society}, 31(3):609--641,
  2018.

\bibitem{conversano2012connected}
Annalisa Conversano and Anand Pillay.
\newblock Connected components of definable groups and o-minimality {I}.
\newblock {\em Advances in Mathematics}, 231(2):605--623, 2012.

\bibitem{Hodges}
Wilfrid Hodges.
\newblock {\em Model theory}, volume~42 of {\em Encycl. Math. Appl.}
\newblock Cambridge: Cambridge University Press, 1993.

\bibitem{hrushovski2008groups}
Ehud Hrushovski, Ya'acov Peterzil, and Anand Pillay.
\newblock Groups, measures, and the {NIP}.
\newblock {\em Journal of the American Mathematical Society}, 21(2):563--596,
  2008.

\bibitem{KimPillay}
Byunghan Kim and Anand Pillay.
\newblock Simple theories.
\newblock {\em Ann. Pure Appl. Logic}, 88(2-3):149--164, 1997.

\bibitem{NeumannCosets}
B.~H. Neumann.
\newblock Groups covered by permutable subsets.
\newblock {\em Journal of the London Mathematical Society}, 29(2):236--248,
  1954.

\bibitem{Newelski}
Ludomir Newelski.
\newblock Bounded orbits and measures on a group.
\newblock {\em Isr. J. Math.}, 187:209--229, 2012.

\bibitem{NPcover}
Ludomir Newelski and Marcin Petrykowski.
\newblock Coverings of groups and types.
\newblock {\em Journal of the London Mathematical Society}, 71(1):1--21, 2005.

\bibitem{NPweak}
Ludomir Newelski and Marcin Petrykowski.
\newblock Weak generic types and coverings of groups {I}.
\newblock {\em Fundamenta Mathematicae}, 191(3):201--225, 2006.

\bibitem{PillayFree}
Anand Pillay.
\newblock Forking in the free group.
\newblock {\em J. Inst. Math. Jussieu}, 7(2):375--389, 2008.

\bibitem{poizat1983groupes}
Bruno Poizat.
\newblock Groupes stables, avec types g{\'e}n{\'e}riques r{\'e}guliers.
\newblock {\em The Journal of Symbolic Logic}, 48(2):339--355, 1983.

\bibitem{SelaStability}
Z.~Sela.
\newblock Diophantine geometry over groups. {VIII}: {Stability}.
\newblock {\em Ann. Math. (2)}, 177(3):787--868, 2013.

\bibitem{SklinosNFCP}
Rizos Sklinos.
\newblock The free group does not have the finite cover property.
\newblock {\em Isr. J. Math.}, 227(2):563--595, 2018.

\bibitem{stonestrom2026f}
Atticus Stonestrom.
\newblock On f-generic types in {NIP} groups.
\newblock {\em Advances in Mathematics}, 503:111222, 2026.

\end{thebibliography}
\end{document}